\documentclass[a4paper,12pt,leqno]{amsart}

\usepackage{latexsym}
\usepackage[all]{xy}

\usepackage{amssymb} 
\usepackage{amsmath} 
\usepackage{color}
\usepackage{comment}
\usepackage{mathdots}
\usepackage{bm}

\definecolor{gray}{gray}{0.7}
\definecolor{Gray}{gray}{0.3}

\usepackage{amscd}

\def\Z{{\mathbb{Z}}}
\def\Q{{\mathbb{Q}}}
\def\C{{\mathbb{C}}}

\DeclareMathOperator{\Hess}{Hess}

\DeclareMathOperator{\id}{id}

\numberwithin{equation}{section}

\theoremstyle{break}
 \newtheorem{theorem}{Theorem}[section]
 \newtheorem{prop}[theorem]{Proposition}
 
 \newtheorem{lemma}[theorem]{Lemma}

 \theoremstyle{definition}
 
 \newtheorem{remark}[theorem]{Remark}
 \newtheorem{example}[theorem]{Example}

\title[Hessenberg varieties and Schubert polynomials]{The cohomology rings of regular nilpotent Hessenberg varieties and Schubert polynomials}

\author{Tatsuya Horiguchi}
\address{Department of Pure and Applied Mathematics,
Graduate School of Information Science and Technology,
Osaka University,
Suita, Osaka, 565-0871, Japan / 
Osaka City University Advanced Mathematical Institute, 3-3-138 Sugimoto, Sumiyoshi-ku, Osaka 558-8585, Japan}
\email{tatsuya.horiguchi0103@gmail.com}
\date{\today} 

\begin{document}

\maketitle

\begin{abstract}
In this paper we study a relation between the cohomology ring of a regular nilpotent Hessenberg variety and Schubert polynomials.
To describe an explicit presentation of the cohomology ring of a regular nilpotent Hessenberg variety, polynomials $f_{i,j}$ were introduced by Abe-Harada-Horiguchi-Masuda.
We show that every polynomial $f_{i,j}$ is an alternating sum of certain Schubert polynomials.
\end{abstract}

\bigskip

\section{Introduction}
 
Let $n$ be a positive integer.
The {\bf (full) flag variety} $\mathcal{F}\ell(\C^n)$ in $\C^n$ is the collection of nested linear subspaces $V_{\bullet}:=(V_1 \subset V_2 \subset \ldots V_n=\C^n)$ where each $V_i$ is an $i$-dimensional subspace in $\C^n$. 
We consider a weakly increasing function $h:\{1,2,\ldots,n \} \to \{1,2,\ldots,n \}$ satisfying $h(j) \geq j \ {\rm for} \ j=1,\ldots,n$. This function is called a {\bf Hessenberg function}.
De~Mari-Procesi-Shayman (\cite{dMS}, \cite{dMPS}) defined a {\bf Hessenberg variety} $\Hess(X,h)$ associated with a linear operator $X: \C^n \to \C^n$ and a Hessenberg function $h:\{1,2,\ldots,n \} \to \{1,2,\ldots,n \}$ as the following subvariety of the flag variety:
\begin{equation} \label{eq:DefHess(X,h)}
\Hess(X,h):=\{V_{\bullet} \in \mathcal{F}\ell(\C^n) \mid XV_i \subset V_{h(i)} \ {\rm for} \ i=1,2,\ldots,n \}.
\end{equation}
We note that if $h(j)=n$ for all $j=1,2,\ldots,n$ or $X$ is the zero matrix, then the corresponding Hessenberg variety coincides with the whole full flag variety $\mathcal{F}\ell(\C^n)$. 
The family of Hessenberg varieties also contains Springer varieties related to geometric representations of Weyl group (\cite{Spr1}, \cite{Spr2}) and Peterson variety related to the quantum cohomology of the flag variety (\cite{Ko}, \cite{R}). 
Recently, it has been found that Hessenberg varieties have surprising connection with other research areas such as hyperplane arrangements (\cite{STy}, \cite{AHMMS}) and graph theory (\cite{SW}, \cite{BC}, \cite{Guay}).

In this paper we concentrate on Hessenberg varieties $\Hess(N,h)$ associated with a regular nilpotent operator $N$ i.e. a matrix whose Jordan form consists of exactly one Jordan block with corresponding eigenvalue equal to $0$. 
The Hessenberg variety $\Hess(N,h)$ is called a {\bf regular nilpotent Hessenberg variety}.
If we take $h(j)=j+1$ for $1 \leq j \leq n-1$ and $h(n)=n$, then the corresponding regular nilpotent Hessenberg variety is called the {\bf Peterson variety}.
Regular nilpotent Hessenberg varieties $\Hess(N,h)$ can be regarded as a (discrete) family of 
subvarieties of the flag variety connecting Peterson variety and the flag variety itself.
The complex dimension of $\Hess(N,h)$ is $\sum_{j=1}^n (h(j)-j)$ (\cite{STy}).
A regular nilpotent Hessenberg variety is singular in general (\cite{Ko}, \cite{IY}).
The cohomology ring of a regular nilpotent Hessenberg variety has been studied from various viewpoints (e.g. \cite{BrionCarrell}, \cite{Ty2}, \cite{HT1}, \cite{Pr1}, \cite{Dre0}, \cite{HHM}, \cite{AHHM}, \cite{AHMMS}).
To describe an explicit presentation of the cohomology\footnote{Throughout this paper (unless explicitly stated otherwise) we work with cohomology with coefficients in $\Q$.} ring of a regular nilpotent Hessenberg variety, polynomials $f_{i,j}$ were introduced in \cite{AHHM} as follows. 
For $1 \leq j \leq i$, we define a polynomial $f_{i,j}$ by
\begin{equation} \label{eq:f_{i,j}}
f_{i,j}:=\sum_{k=1}^{j} \big( \prod_{\ell=j+1}^i (x_k-x_{\ell})\big)x_k.
\end{equation}
Here, we take by convention $\prod_{\ell=j+1}^i (x_k-x_{\ell})=1$ whenever $j=i$.
Then from the result of \cite{AHHM}, the following isomorphism as $\Q$-algebras holds
\begin{equation} \label{eq:AHHM} 
H^*(\Hess(N,h)) \cong \Q[x_1,\ldots,x_n]/(f_{h(1),1},f_{h(2),2},\ldots,f_{h(n),n}).
\end{equation}
Our main theorem is the following. 

\begin{theorem} \label{theorem:main1}
Let $j,i$ be positive integers with $1 \leq j<i \leq n$. 
Let $f_{i-1,j}$ be the polynomial in $\eqref{eq:f_{i,j}}$ and $\mathfrak{S}_{w}$ the Schubert polynomial for a permutation $w$ in  the symmetric group $S_n$. 
Then we have
\begin{equation} \label{eq:main1}
f_{i-1,j}=\sum_{k=1}^{i-j} (-1)^{k-1} \mathfrak{S}_{w_k^{(i,j)}}
\end{equation}
where $w_k^{(i,j)}$ for $1 \leq k \leq i-j$ is a permutation in $S_n$ defined by
\begin{equation} \label{eq:w_k^{(i,j)}}
w_k^{(i,j)}:=(s_{i-k}s_{i-k-1} \ldots s_{j})(s_{i-k+1}s_{i-k+2} \ldots s_{i-1}).
\end{equation}
Here, $s_r$ denotes the transposition of $r$ and $r+1$ for $r=1,2,\ldots,n-1$
and we take by convention $(s_{i-k+1}s_{i-k+2} \ldots s_{i-1})=\id$ whenever $k=1$.
\end{theorem}

We can interpret the equality \eqref{eq:main1} in Theorem~\ref{theorem:main1} from a geometric viewpoint under the circumstances of having a codimension one Hessenberg variety $\Hess(N,h')$ in the original Hessenberg variety $\Hess(N,h)$. 
We will discuss more details in Section~\ref{sect:Geometric observation of the main theorem}.

\bigskip

\section{Divided difference operator} \label{sect:Divided difference operator}

In this section, we observe a new property of polynomials $f_{i,j}$ in \eqref{eq:f_{i,j}} related with the divided difference operator defined by Bernstein-Gelfand-Gelfand and Demazure. 
This is the key property for the proof of the main theorem.
We first recall the definition of the divided difference operator and the Schubert polynomials.
For general reference, see \cite{Fult}.

Let $f$ be a polynomial in $\Z[x_1,\ldots,x_n]$ and $s_i$ the transposition of $i$ and $i+1$ for any $i=1,2,\ldots, n-1$.
Let $s_i(f)$ denote the result of interchanging $x_i$ and $x_{i+1}$ in $f$.
Then the {\bf divided difference operator} $\partial_i$ on the polynomial ring $\Z[x_1,\ldots,x_n]$ is defined by the formula 
\begin{equation} \label{eq:DDO}
\partial_i (f):=\frac{f-s_i(f)}{x_i-x_{i+1}}.
\end{equation}
Since $f-s_i(f)$ is divisible by $x_i-x_{i+1}$, $\partial_i (f)$ is always a polynomial.
If $f$ is homogeneous of degree $d$, then $\partial_i (f)$ is homogeneous of degree $d-1$.

For a reduced expression $u=s_{i_1}s_{i_2}\ldots s_{i_r}$, we set $\partial_{u}=\partial_{i_1}\partial_{i_2}\ldots \partial_{i_r}$. 
Since the divided difference operators satisfy the relations $\partial_{i}\partial_{i+1}\partial_{i}=\partial_{i+1}\partial_{i}\partial_{i+1}$ and $\partial_{i}^2 = 0$, the operator $\partial_{u}$ is independent of the choice of reduced expressions for $u$. 
The {\bf Schubert polynomial} $\mathfrak{S}_{w}$ for a permutation $w$ in the symmetric group $S_n$ is defined as follows. 
For $w_0 = [n, n-1, \ldots , 1] \in S_n$ the permutation of the longest length in one-line notation, we define
$$
\mathfrak{S}_{w_0}=\mathfrak{S}_{w_0}(x_1,\ldots,x_n) = x_1^{n-1}x_2^{n-2}\cdots x_{n-1}.
$$
For general permutation $w$ in $S_n$, write $w = w_0 s_{i_1} s_{i_2} \cdots s_{i_r}$ with $\ell(w_0 s_{i_1} s_{i_2} \cdots s_{i_p}) = \ell(w_0)-p$ for $1 \leq p \leq r$.
Then the Schubert polynomial is inductively defined by
\begin{align*}
\mathfrak{S}_{w}=&\mathfrak{S}_{w}(x_1,\ldots,x_n)= \partial_{i_r} \circ \ldots \circ \partial_{i_2} \circ\partial_{i_1} (\mathfrak{S}_{w_0}(x_1,\ldots,x_n)) \\
=& \partial_{i_r} \circ \ldots \circ \partial_{i_2} \circ\partial_{i_1} (x_1^{n-1}x_2^{n-2}\cdots x_{n-1}).
\end{align*}
In general, Schubert polynomials have the following property 
\begin{equation}\label{eq:SchubertDDO}
\partial_i\mathfrak{S}_w=
\begin{cases}
    \mathfrak{S}_{ws_i} \ &{\rm if } \ w(i)>w(i+1),  \\
    0                           &{\rm if } \ w(i)<w(i+1).
  \end{cases}
\end{equation}
Note that the Schubert polynomial $\mathfrak{S}_{w}$ is a homogeneous polynomial in $\Z[x_1,\ldots,x_{n-1}]$ of degree $\ell(w)$ which is the number of inversions in $w$, called the {\bf length} of $w$, i.e.,
$$ 
\ell(w)=\# \{j<i \mid w(j)>w(i) \}.
$$ 
Schubert polynomials have an important property that $\mathfrak{S}_{w}$ is in fact independent of $n$ in the following sense. 
For $w \in S_n$ and $m \geq n$, we define $w^{(m)}\in S_m$ by $w^{(m)}(i)=w(i)$ for $1 \leq i \leq n$ and $w^{(m)}(i)=i$ for $n+1 \leq i \leq m$.
Then we have
\begin{equation} \label{eq:stabilitySchubertPolynomial}
\mathfrak{S}_{w}=\mathfrak{S}_{w^{(m)}}.
\end{equation}

The following proposition is the key property for the proof of the main theorem.

\begin{prop}\label{prop:DDO}
Let $j,i$ be positive integers with $j<i$. 
Let $f_{i,j}$ be the polynomial in $\eqref{eq:f_{i,j}}$ and $\partial_i$ the divided difference operator in $\eqref{eq:DDO}$. 
Then we have
\begin{align}
\partial_j (f_{i,j})=f_{i,j+1}. \label{eq:HoriguchiDDO} 
\end{align}
\end{prop}

\begin{proof}
Since
\begin{align*}
f_{i,j}&=\sum_{k=1}^{j-1} \big( \prod_{\ell=j+2}^i (x_k-x_{\ell})\big)(x_k-x_{j+1})x_k + \big( \prod_{\ell=j+2}^i (x_{j}-x_{\ell})\big)(x_j-x_{j+1})x_{j}, \\
s_j (f_{i,j})&=\sum_{k=1}^{j-1} \big( \prod_{\ell=j+2}^i (x_k-x_{\ell})\big)(x_k-x_j)x_k + \big( \prod_{\ell=j+2}^i (x_{j+1}-x_{\ell})\big)(x_{j+1}-x_{j})x_{j+1}, 
\end{align*}
we have
\begin{align*}
f_{i,j}-s_j (f_{i,j})=(x_j-x_{j+1})\big(\sum_{k=1}^{j+1} \big( \prod_{\ell=j+2}^i (x_k-x_{\ell})\big)x_k \big). 
\end{align*}
Hence, we obtain $\partial_j (f_{i,j})=f_{i,j+1}$.
\end{proof}

\begin{remark}
It also follows that
\begin{align}
\partial_i (f_{i,j})=-f_{i-1,j}. \label{eq:AbeDDO}
\end{align}
The proof is similar to the proof of \eqref{eq:HoriguchiDDO}. 
From \eqref{eq:HoriguchiDDO} together with \eqref{eq:AbeDDO}, we see that every polynomial $f_{i,j}$ for $1 \leq j \leq i \leq n$ is obtained from the single polynomial $f_{n,1}$ by using the divided difference operator.
More concretely, we set
\begin{align*}
F_n:=f_{n,1}=(x_1-x_n)(x_1-x_{n-1})\cdots(x_1-x_2)x_1. 
\end{align*}
Then we obtain
\begin{equation*} 
f_{i,j}=((-1)^{n-i}\partial_{i+1} \partial_{i+2} \ldots \partial_n)(\partial_{j-1} \partial_{j-2} \ldots \partial_1)(F_n).
\end{equation*}
\end{remark}

\bigskip

\section{Proof of Theorem~\ref{theorem:main1}} \label{sect:Proof of Theorem}

In this section we prove Theorem~\ref{theorem:main1}.
To do that, we need Monk's formula.

\begin{theorem} [Monk's formula \cite{Monk}, see also \cite{Fult} p.180] 
Let $\mathfrak{S}_{w}$ be the Schubert polynomial for $w\in S_n$ and $s_r$ the transposition of $r$ and $r+1$. 
Then we have
\begin{equation}\label{eq:Monk}
\mathfrak{S}_{s_r} \cdot \mathfrak{S}_{w}=\sum \mathfrak{S}_{wt_{p \, q}}
\end{equation}
where $t_{p \, q}$ is the transposition interchanging values of $p$ and $q$, and the sum is over all $1 \leq p \leq r  < q$ such that $w(p)<w(q)$ and $w(i)$ is not in the interval $(w(p),w(q))$ for any $i$ in the interval $(p,q)$. \\
\end{theorem}

Using Monk's formula, we first prove the following proposition which is the case $i=n$ and $j=1$ of Theorem~\ref{theorem:main1}.

\begin{prop} \label{prop:j=1}
Let $n>1$ and $f_{n-1,1}$ the polynomial in $\eqref{eq:f_{i,j}}$.
Let $\mathfrak{S}_{w}$ be the Schubert polynomial for $w\in S_n$. 
Then we have
\begin{equation*}
f_{n-1,1}=\sum_{k=1}^{n-1} (-1)^{k-1} \mathfrak{S}_{w_k^{(n,1)}}
\end{equation*}
where $w_k^{(n,1)}\in S_n$ is the permutation defined in \eqref{eq:w_k^{(i,j)}}.
\end{prop}

\begin{proof}
We prove the proposition by induction on $n$. 
For the base case $n=2$, it holds because $f_{1,1}=x_1=\mathfrak{S}_{s_1}=\mathfrak{S}_{w_1^{(2,1)}}$.
Now we assume $n>2$ and the following equality 
$$
f_{n-2,1}=\sum_{k=1}^{n-2} (-1)^{k-1} \mathfrak{S}_{w_k^{(n-1,1)}}.
$$
Since we have from the definition \eqref{eq:f_{i,j}} that
\begin{align*}
f_{n-1,1}=(x_1-x_{n-1}) f_{n-2,1}&=\big(\mathfrak{S}_{s_1}+\mathfrak{S}_{s_{n-2}}-\mathfrak{S}_{s_{n-1}}\big) \big( \sum_{k=1}^{n-2} (-1)^{k-1} \mathfrak{S}_{w_k^{(n-1,1)}} \big), 
\end{align*}
it is enough to prove the following equality
\begin{equation} \label{eq:claim-prop:j=1}
\big(\mathfrak{S}_{s_1}+\mathfrak{S}_{s_{n-2}}-\mathfrak{S}_{s_{n-1}}\big) \big( \sum_{k=1}^{n-2} (-1)^{k-1} \mathfrak{S}_{w_k^{(n-1,1)}} \big)
=\sum_{k=1}^{n-1} (-1)^{k-1} \mathfrak{S}_{w_k^{(n,1)}}.
\end{equation}
We prove the equality \eqref{eq:claim-prop:j=1} using Monk's formula \eqref{eq:Monk}.\\
{\bf \underline{Case(i)}}:
Using Monk's formula \eqref{eq:Monk}, the calculus of $\mathfrak{S}_{s_1} \cdot \mathfrak{S}_{w_k^{(n-1,1)}}$ is as follows 
\begin{align*} 
\mathfrak{S}_{s_1} \cdot \mathfrak{S}_{w_1^{(n-1,1)}}&=\mathfrak{S}_{w_1^{(n-1,1)}t_{1 \, n}}, \\
\mathfrak{S}_{s_1} \cdot \mathfrak{S}_{w_k^{(n-1,1)}}&=\mathfrak{S}_{w_k^{(n-1,1)}t_{1 \, n-k}} \ \ \ \ \ {\rm if} \ k \neq 1.
\end{align*} 
{\bf \underline{Case(ii)}}:
Using Monk's formula \eqref{eq:Monk}, the calculus of $\mathfrak{S}_{s_{n-2}} \cdot \mathfrak{S}_{w_k^{(n-1,1)}}$ is as follows
\begin{align*} 
\mathfrak{S}_{s_{n-2}} \cdot \mathfrak{S}_{w_1^{(n-1,1)}}&=\mathfrak{S}_{w_1^{(n-1,1)}t_{1 \, n}}+\mathfrak{S}_{w_1^{(n-1,1)}t_{n-2 \, n-1}}, \\
\mathfrak{S}_{s_{n-2}} \cdot \mathfrak{S}_{w_k^{(n-1,1)}}&=\mathfrak{S}_{w_k^{(n-1,1)}t_{n-k-1 \, n-1}}+\mathfrak{S}_{w_k^{(n-1,1)}t_{n-2 \, n}} \ \ \ \ \ {\rm if} \ k \neq 1, \ n-2, \\
\mathfrak{S}_{s_{n-2}} \cdot \mathfrak{S}_{w_{n-2}^{(n-1,1)}}&=\mathfrak{S}_{w_{n-2}^{(n-1,1)}t_{n-2 \, n}}. 
\end{align*} 
{\bf \underline{Case(iii)}}:
Using Monk's formula \eqref{eq:Monk}, the calculus of $\mathfrak{S}_{s_{n-1}} \cdot \mathfrak{S}_{w_k^{(n-1,1)}}$ is as follows
\begin{align*} 
\mathfrak{S}_{s_{n-1}} \cdot \mathfrak{S}_{w_1^{(n-1,1)}}&=\mathfrak{S}_{w_1^{(n-1,1)}t_{1 \, n}}+\mathfrak{S}_{w_1^{(n-1,1)}t_{n-1 \, n}}, \\
\mathfrak{S}_{s_{n-1}} \cdot \mathfrak{S}_{w_k^{(n-1,1)}}&=\mathfrak{S}_{w_k^{(n-1,1)}t_{n-2 \, n}}+\mathfrak{S}_{w_k^{(n-1,1)}t_{n-1 \, n}} \ \ \ \ \ {\rm if} \ k \neq 1. 
\end{align*} 
From Case(i),(ii),(iii) together with equalities $w_1^{(n-1,1)}t_{1 \, n}=w_1^{(n,1)}$ and $w_k^{(n-1,1)}t_{n-1 \, n}=w_{k+1}^{(n,1)}$ for $1 \leq k \leq n-2$, the left hand side of \eqref{eq:claim-prop:j=1} reduces to
\begin{align*} 
&\big( \mathfrak{S}_{w_1^{(n,1)}} - \sum_{k=1}^{n-3} (-1)^{k-1}\mathfrak{S}_{w_{k+1}^{(n-1,1)}t_{1 \, n-k-1}} \big) +\big( \sum_{k=1}^{n-3} (-1)^{k-1}\mathfrak{S}_{w_k^{(n-1,1)}t_{n-k-1 \, n-1}} \big) \\
&-\big( \mathfrak{S}_{w_2^{(n,1)}}+\sum_{k=2}^{n-2} (-1)^{k-1}\mathfrak{S}_{w_{k+1}^{(n,1)}} \big).
\end{align*} 
However, since $w_{k+1}^{(n-1,1)}t_{1 \, n-k-1}=w_{k}^{(n-1,1)}t_{n-k-1 \, n-1}$ for $1 \leq k \leq n-3$, the above expression is equal to the right hand side of \eqref{eq:claim-prop:j=1}.
Therefore, we obtain \eqref{eq:claim-prop:j=1}. 
This completes the induction step and proves the proposition.
\end{proof}

\begin{proof}[Proof of Theorem~$\ref{theorem:main1}$]
We now prove Theorem~$\ref{theorem:main1}$ by induction on $j$. 
For the base case $j=1$, it holds from Proposition~\ref{prop:j=1} together with the property \eqref{eq:stabilitySchubertPolynomial} of Schubert polynomials.
Now we assume $j>1$ and the following equality 
\begin{equation*}
f_{i-1,j-1}=\sum_{k=1}^{i-j+1} (-1)^{k-1} \mathfrak{S}_{w_k^{(i,j-1)}}.
\end{equation*}
Then since
\begin{align*}
w_k^{(i,j-1)}(j-1)>w_k^{(i,j-1)}(j) \ \ \ \ \ &{\rm if} \  1 \leq k \leq i-j, \\ 
w_k^{(i,j-1)}(j-1)<w_k^{(i,j-1)}(j) \ \ \ \ \ &{\rm if} \  k=i-j+1, 
\end{align*}
we have from \eqref{eq:SchubertDDO} that
\begin{equation*}
\partial_{j-1}(\mathfrak{S}_{w_k^{(i,j-1)}})=
\begin{cases}
    \mathfrak{S}_{w_k^{(i,j-1)}s_{j-1}} \ &{\rm if } \ 1 \leq k \leq i-j,  \\
    0                                       &{\rm if } \ k=i-j+1. 
  \end{cases}
\end{equation*}
Therefore, using \eqref{eq:HoriguchiDDO}, we obtain
\begin{align*}
f_{i-1,j}&=\partial_{j-1}(f_{i-1,j-1})
=\partial_{j-1} \big( \sum_{k=1}^{i-j+1} (-1)^{k-1} \mathfrak{S}_{w_k^{(i,j-1)}} \big) \\
&=\sum_{k=1}^{i-j} (-1)^{k-1}\mathfrak{S}_{w_k^{(i,j-1)}s_{j-1}}
=\sum_{k=1}^{i-j} (-1)^{k-1}\mathfrak{S}_{w_k^{(i,j)}}.
\end{align*}
This completes the induction step and proves Theorem~$\ref{theorem:main1}$.
\end{proof}

\smallskip

\section{Geometric meaning of Theorem~\ref{theorem:main1}} \label{sect:Geometric observation of the main theorem}

In this section we observe a geometric meaning of Theorem~\ref{theorem:main1}.
Throughout this section we use the notation 
$$
[n]:=\{1,2,\ldots,n \}.
$$
Recall that a Hessenberg function $h: [n] \to [n]$ is a weakly increasing function satisfying $h(j) \geq j$ for $j\in [n]$. 
We denote a Hessenberg function $h$ by listing its values in sequence, i.e. 
$$
h=(h(1),h(2),\ldots,h(n)).
$$ 
We often regard a Hessenberg function as a configuration of boxes on a square grid of
size $n \times n$ whose shaded boxes correspond to the boxes in the position $(i,j)$ for $i, j\in [n]$ and $i \leq h(j)$.

\begin{example} \label{example:h=(3,3,4,5,5)}
Let $n=5$.
A function $h=(3,3,4,5,5)$ is a Hessenberg function and the corresponding configuration of boxes on a square grid of size $5 \times 5$ is given in Figure~\ref{picture:h=(3,3,4,5,5)}. 
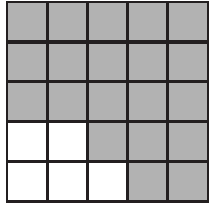
\begin{figure}[h]
\begin{center}
\begin{picture}(75,65)
\put(0,63){\colorbox{gray}}
\put(0,67){\colorbox{gray}}
\put(0,72){\colorbox{gray}}
\put(4,63){\colorbox{gray}}
\put(4,67){\colorbox{gray}}
\put(4,72){\colorbox{gray}}
\put(9,63){\colorbox{gray}}
\put(9,67){\colorbox{gray}}
\put(9,72){\colorbox{gray}}

\put(15,63){\colorbox{gray}}
\put(15,67){\colorbox{gray}}
\put(15,72){\colorbox{gray}}
\put(19,63){\colorbox{gray}}
\put(19,67){\colorbox{gray}}
\put(19,72){\colorbox{gray}}
\put(24,63){\colorbox{gray}}
\put(24,67){\colorbox{gray}}
\put(24,72){\colorbox{gray}}

\put(30,63){\colorbox{gray}}
\put(30,67){\colorbox{gray}}
\put(30,72){\colorbox{gray}}
\put(34,63){\colorbox{gray}}
\put(34,67){\colorbox{gray}}
\put(34,72){\colorbox{gray}}
\put(39,63){\colorbox{gray}}
\put(39,67){\colorbox{gray}}
\put(39,72){\colorbox{gray}}

\put(45,63){\colorbox{gray}}
\put(45,67){\colorbox{gray}}
\put(45,72){\colorbox{gray}}
\put(49,63){\colorbox{gray}}
\put(49,67){\colorbox{gray}}
\put(49,72){\colorbox{gray}}
\put(54,63){\colorbox{gray}}
\put(54,67){\colorbox{gray}}
\put(54,72){\colorbox{gray}}

\put(60,63){\colorbox{gray}}
\put(60,67){\colorbox{gray}}
\put(60,72){\colorbox{gray}}
\put(64,63){\colorbox{gray}}
\put(64,67){\colorbox{gray}}
\put(64,72){\colorbox{gray}}
\put(69,63){\colorbox{gray}}
\put(69,67){\colorbox{gray}}
\put(69,72){\colorbox{gray}}

\put(0,48){\colorbox{gray}}
\put(0,52){\colorbox{gray}}
\put(0,57){\colorbox{gray}}
\put(4,48){\colorbox{gray}}
\put(4,52){\colorbox{gray}}
\put(4,57){\colorbox{gray}}
\put(9,48){\colorbox{gray}}
\put(9,52){\colorbox{gray}}
\put(9,57){\colorbox{gray}}

\put(15,48){\colorbox{gray}}
\put(15,52){\colorbox{gray}}
\put(15,57){\colorbox{gray}}
\put(19,48){\colorbox{gray}}
\put(19,52){\colorbox{gray}}
\put(19,57){\colorbox{gray}}
\put(24,48){\colorbox{gray}}
\put(24,52){\colorbox{gray}}
\put(24,57){\colorbox{gray}}

\put(30,48){\colorbox{gray}}
\put(30,52){\colorbox{gray}}
\put(30,57){\colorbox{gray}}
\put(34,48){\colorbox{gray}}
\put(34,52){\colorbox{gray}}
\put(34,57){\colorbox{gray}}
\put(39,48){\colorbox{gray}}
\put(39,52){\colorbox{gray}}
\put(39,57){\colorbox{gray}}

\put(45,48){\colorbox{gray}}
\put(45,52){\colorbox{gray}}
\put(45,57){\colorbox{gray}}
\put(49,48){\colorbox{gray}}
\put(49,52){\colorbox{gray}}
\put(49,57){\colorbox{gray}}
\put(54,48){\colorbox{gray}}
\put(54,52){\colorbox{gray}}
\put(54,57){\colorbox{gray}}

\put(60,48){\colorbox{gray}}
\put(60,52){\colorbox{gray}}
\put(60,57){\colorbox{gray}}
\put(64,48){\colorbox{gray}}
\put(64,52){\colorbox{gray}}
\put(64,57){\colorbox{gray}}
\put(69,48){\colorbox{gray}}
\put(69,52){\colorbox{gray}}
\put(69,57){\colorbox{gray}}

\put(0,33){\colorbox{gray}}
\put(0,37){\colorbox{gray}}
\put(0,42){\colorbox{gray}}
\put(4,33){\colorbox{gray}}
\put(4,37){\colorbox{gray}}
\put(4,42){\colorbox{gray}}
\put(9,33){\colorbox{gray}}
\put(9,37){\colorbox{gray}}
\put(9,42){\colorbox{gray}}

\put(15,33){\colorbox{gray}}
\put(15,37){\colorbox{gray}}
\put(15,42){\colorbox{gray}}
\put(19,33){\colorbox{gray}}
\put(19,37){\colorbox{gray}}
\put(19,42){\colorbox{gray}}
\put(24,33){\colorbox{gray}}
\put(24,37){\colorbox{gray}}
\put(24,42){\colorbox{gray}}

\put(30,33){\colorbox{gray}}
\put(30,37){\colorbox{gray}}
\put(30,42){\colorbox{gray}}
\put(34,33){\colorbox{gray}}
\put(34,37){\colorbox{gray}}
\put(34,42){\colorbox{gray}}
\put(39,33){\colorbox{gray}}
\put(39,37){\colorbox{gray}}
\put(39,42){\colorbox{gray}}

\put(45,33){\colorbox{gray}}
\put(45,37){\colorbox{gray}}
\put(45,42){\colorbox{gray}}
\put(49,33){\colorbox{gray}}
\put(49,37){\colorbox{gray}}
\put(49,42){\colorbox{gray}}
\put(54,33){\colorbox{gray}}
\put(54,37){\colorbox{gray}}
\put(54,42){\colorbox{gray}}

\put(60,33){\colorbox{gray}}
\put(60,37){\colorbox{gray}}
\put(60,42){\colorbox{gray}}
\put(64,33){\colorbox{gray}}
\put(64,37){\colorbox{gray}}
\put(64,42){\colorbox{gray}}
\put(69,33){\colorbox{gray}}
\put(69,37){\colorbox{gray}}
\put(69,42){\colorbox{gray}}

%
%
\put(30,18){\colorbox{gray}}
\put(30,22){\colorbox{gray}}
\put(30,27){\colorbox{gray}}
\put(34,18){\colorbox{gray}}
\put(34,22){\colorbox{gray}}
\put(34,27){\colorbox{gray}}
\put(39,18){\colorbox{gray}}
\put(39,22){\colorbox{gray}}
\put(39,27){\colorbox{gray}}

\put(45,18){\colorbox{gray}}
\put(45,22){\colorbox{gray}}
\put(45,27){\colorbox{gray}}
\put(49,18){\colorbox{gray}}
\put(49,22){\colorbox{gray}}
\put(49,27){\colorbox{gray}}
\put(54,18){\colorbox{gray}}
\put(54,22){\colorbox{gray}}
\put(54,27){\colorbox{gray}}

\put(60,18){\colorbox{gray}}
\put(60,22){\colorbox{gray}}
\put(60,27){\colorbox{gray}}
\put(64,18){\colorbox{gray}}
\put(64,22){\colorbox{gray}}
\put(64,27){\colorbox{gray}}
\put(69,18){\colorbox{gray}}
\put(69,22){\colorbox{gray}}
\put(69,27){\colorbox{gray}}

%
%
%
\put(45,3){\colorbox{gray}}
\put(45,7){\colorbox{gray}}
\put(45,12){\colorbox{gray}}
\put(49,3){\colorbox{gray}}
\put(49,7){\colorbox{gray}}
\put(49,12){\colorbox{gray}}
\put(54,3){\colorbox{gray}}
\put(54,7){\colorbox{gray}}
\put(54,12){\colorbox{gray}}

\put(60,3){\colorbox{gray}}
\put(60,7){\colorbox{gray}}
\put(60,12){\colorbox{gray}}
\put(64,3){\colorbox{gray}}
\put(64,7){\colorbox{gray}}
\put(64,12){\colorbox{gray}}
\put(69,3){\colorbox{gray}}
\put(69,7){\colorbox{gray}}
\put(69,12){\colorbox{gray}}

\put(0,0){\framebox(15,15)}
\put(15,0){\framebox(15,15)}
\put(30,0){\framebox(15,15)}
\put(45,0){\framebox(15,15)}
\put(60,0){\framebox(15,15)}
\put(0,15){\framebox(15,15)}
\put(15,15){\framebox(15,15)}
\put(30,15){\framebox(15,15)}
\put(45,15){\framebox(15,15)}
\put(60,15){\framebox(15,15)}
\put(0,30){\framebox(15,15)}
\put(15,30){\framebox(15,15)}
\put(30,30){\framebox(15,15)}
\put(45,30){\framebox(15,15)}
\put(60,30){\framebox(15,15)}
\put(0,45){\framebox(15,15)}
\put(15,45){\framebox(15,15)}
\put(30,45){\framebox(15,15)}
\put(45,45){\framebox(15,15)}
\put(60,45){\framebox(15,15)}
\put(0,60){\framebox(15,15)}
\put(15,60){\framebox(15,15)}
\put(30,60){\framebox(15,15)}
\put(45,60){\framebox(15,15)}
\put(60,60){\framebox(15,15)}
\end{picture}
\end{center}
\caption{the configuration of shaded boxes for $h=(3,3,4,5,5)$}
\label{picture:h=(3,3,4,5,5)}
\end{figure} 
\end{example}

An $(i,j)$-th box of a Hessenberg function is a {\bf corner} if there is neither a shaded box in $(i+1,j)$ nor in $(i,j-1)$. 
Let $h:[n] \to [n]$ be a Hessenberg function with $(i,j)$-th box as a corner with $i>j$.
We define a Hessenberg function $h':[n] \to [n]$ by removing $(i,j)$-th box of $h$ (see Figure~\ref{picture:h and h'}). 
More precisely, $h'$ is defined by
\begin{align*}
&h'(k)=h(k) \ \ \ \ \ \ \ \ \ \ \ \ \ \ \ \ \ \ \ \ {\rm if} \ k\neq j, \\
&h'(j)=h(j)-1=i-1.
\end{align*}

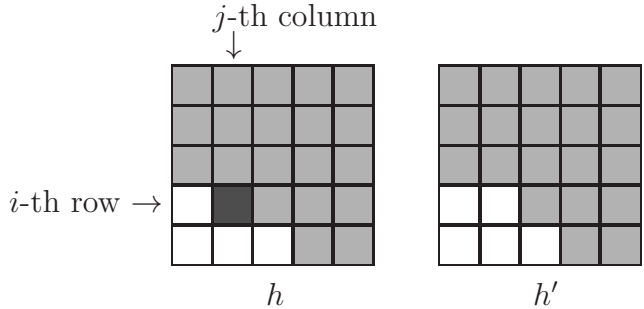
\begin{figure}[h]
\begin{center}
\begin{picture}(175,100)
\put(0,63){\colorbox{gray}}
\put(0,67){\colorbox{gray}}
\put(0,72){\colorbox{gray}}
\put(4,63){\colorbox{gray}}
\put(4,67){\colorbox{gray}}
\put(4,72){\colorbox{gray}}
\put(9,63){\colorbox{gray}}
\put(9,67){\colorbox{gray}}
\put(9,72){\colorbox{gray}}

\put(15,63){\colorbox{gray}}
\put(15,67){\colorbox{gray}}
\put(15,72){\colorbox{gray}}
\put(19,63){\colorbox{gray}}
\put(19,67){\colorbox{gray}}
\put(19,72){\colorbox{gray}}
\put(24,63){\colorbox{gray}}
\put(24,67){\colorbox{gray}}
\put(24,72){\colorbox{gray}}

\put(30,63){\colorbox{gray}}
\put(30,67){\colorbox{gray}}
\put(30,72){\colorbox{gray}}
\put(34,63){\colorbox{gray}}
\put(34,67){\colorbox{gray}}
\put(34,72){\colorbox{gray}}
\put(39,63){\colorbox{gray}}
\put(39,67){\colorbox{gray}}
\put(39,72){\colorbox{gray}}

\put(45,63){\colorbox{gray}}
\put(45,67){\colorbox{gray}}
\put(45,72){\colorbox{gray}}
\put(49,63){\colorbox{gray}}
\put(49,67){\colorbox{gray}}
\put(49,72){\colorbox{gray}}
\put(54,63){\colorbox{gray}}
\put(54,67){\colorbox{gray}}
\put(54,72){\colorbox{gray}}

\put(60,63){\colorbox{gray}}
\put(60,67){\colorbox{gray}}
\put(60,72){\colorbox{gray}}
\put(64,63){\colorbox{gray}}
\put(64,67){\colorbox{gray}}
\put(64,72){\colorbox{gray}}
\put(69,63){\colorbox{gray}}
\put(69,67){\colorbox{gray}}
\put(69,72){\colorbox{gray}}

\put(0,48){\colorbox{gray}}
\put(0,52){\colorbox{gray}}
\put(0,57){\colorbox{gray}}
\put(4,48){\colorbox{gray}}
\put(4,52){\colorbox{gray}}
\put(4,57){\colorbox{gray}}
\put(9,48){\colorbox{gray}}
\put(9,52){\colorbox{gray}}
\put(9,57){\colorbox{gray}}

\put(15,48){\colorbox{gray}}
\put(15,52){\colorbox{gray}}
\put(15,57){\colorbox{gray}}
\put(19,48){\colorbox{gray}}
\put(19,52){\colorbox{gray}}
\put(19,57){\colorbox{gray}}
\put(24,48){\colorbox{gray}}
\put(24,52){\colorbox{gray}}
\put(24,57){\colorbox{gray}}

\put(30,48){\colorbox{gray}}
\put(30,52){\colorbox{gray}}
\put(30,57){\colorbox{gray}}
\put(34,48){\colorbox{gray}}
\put(34,52){\colorbox{gray}}
\put(34,57){\colorbox{gray}}
\put(39,48){\colorbox{gray}}
\put(39,52){\colorbox{gray}}
\put(39,57){\colorbox{gray}}

\put(45,48){\colorbox{gray}}
\put(45,52){\colorbox{gray}}
\put(45,57){\colorbox{gray}}
\put(49,48){\colorbox{gray}}
\put(49,52){\colorbox{gray}}
\put(49,57){\colorbox{gray}}
\put(54,48){\colorbox{gray}}
\put(54,52){\colorbox{gray}}
\put(54,57){\colorbox{gray}}

\put(60,48){\colorbox{gray}}
\put(60,52){\colorbox{gray}}
\put(60,57){\colorbox{gray}}
\put(64,48){\colorbox{gray}}
\put(64,52){\colorbox{gray}}
\put(64,57){\colorbox{gray}}
\put(69,48){\colorbox{gray}}
\put(69,52){\colorbox{gray}}
\put(69,57){\colorbox{gray}}

\put(0,33){\colorbox{gray}}
\put(0,37){\colorbox{gray}}
\put(0,42){\colorbox{gray}}
\put(4,33){\colorbox{gray}}
\put(4,37){\colorbox{gray}}
\put(4,42){\colorbox{gray}}
\put(9,33){\colorbox{gray}}
\put(9,37){\colorbox{gray}}
\put(9,42){\colorbox{gray}}

\put(15,33){\colorbox{gray}}
\put(15,37){\colorbox{gray}}
\put(15,42){\colorbox{gray}}
\put(19,33){\colorbox{gray}}
\put(19,37){\colorbox{gray}}
\put(19,42){\colorbox{gray}}
\put(24,33){\colorbox{gray}}
\put(24,37){\colorbox{gray}}
\put(24,42){\colorbox{gray}}

\put(30,33){\colorbox{gray}}
\put(30,37){\colorbox{gray}}
\put(30,42){\colorbox{gray}}
\put(34,33){\colorbox{gray}}
\put(34,37){\colorbox{gray}}
\put(34,42){\colorbox{gray}}
\put(39,33){\colorbox{gray}}
\put(39,37){\colorbox{gray}}
\put(39,42){\colorbox{gray}}

\put(45,33){\colorbox{gray}}
\put(45,37){\colorbox{gray}}
\put(45,42){\colorbox{gray}}
\put(49,33){\colorbox{gray}}
\put(49,37){\colorbox{gray}}
\put(49,42){\colorbox{gray}}
\put(54,33){\colorbox{gray}}
\put(54,37){\colorbox{gray}}
\put(54,42){\colorbox{gray}}

\put(60,33){\colorbox{gray}}
\put(60,37){\colorbox{gray}}
\put(60,42){\colorbox{gray}}
\put(64,33){\colorbox{gray}}
\put(64,37){\colorbox{gray}}
\put(64,42){\colorbox{gray}}
\put(69,33){\colorbox{gray}}
\put(69,37){\colorbox{gray}}
\put(69,42){\colorbox{gray}}

%
\put(15,18){\colorbox{Gray}}
\put(15,22){\colorbox{Gray}}
\put(15,27){\colorbox{Gray}}
\put(19,18){\colorbox{Gray}}
\put(19,22){\colorbox{Gray}}
\put(19,27){\colorbox{Gray}}
\put(24,18){\colorbox{Gray}}
\put(24,22){\colorbox{Gray}}
\put(24,27){\colorbox{Gray}}

\put(30,18){\colorbox{gray}}
\put(30,22){\colorbox{gray}}
\put(30,27){\colorbox{gray}}
\put(34,18){\colorbox{gray}}
\put(34,22){\colorbox{gray}}
\put(34,27){\colorbox{gray}}
\put(39,18){\colorbox{gray}}
\put(39,22){\colorbox{gray}}
\put(39,27){\colorbox{gray}}

\put(45,18){\colorbox{gray}}
\put(45,22){\colorbox{gray}}
\put(45,27){\colorbox{gray}}
\put(49,18){\colorbox{gray}}
\put(49,22){\colorbox{gray}}
\put(49,27){\colorbox{gray}}
\put(54,18){\colorbox{gray}}
\put(54,22){\colorbox{gray}}
\put(54,27){\colorbox{gray}}

\put(60,18){\colorbox{gray}}
\put(60,22){\colorbox{gray}}
\put(60,27){\colorbox{gray}}
\put(64,18){\colorbox{gray}}
\put(64,22){\colorbox{gray}}
\put(64,27){\colorbox{gray}}
\put(69,18){\colorbox{gray}}
\put(69,22){\colorbox{gray}}
\put(69,27){\colorbox{gray}}

%
%
%
\put(45,3){\colorbox{gray}}
\put(45,7){\colorbox{gray}}
\put(45,12){\colorbox{gray}}
\put(49,3){\colorbox{gray}}
\put(49,7){\colorbox{gray}}
\put(49,12){\colorbox{gray}}
\put(54,3){\colorbox{gray}}
\put(54,7){\colorbox{gray}}
\put(54,12){\colorbox{gray}}

\put(60,3){\colorbox{gray}}
\put(60,7){\colorbox{gray}}
\put(60,12){\colorbox{gray}}
\put(64,3){\colorbox{gray}}
\put(64,7){\colorbox{gray}}
\put(64,12){\colorbox{gray}}
\put(69,3){\colorbox{gray}}
\put(69,7){\colorbox{gray}}
\put(69,12){\colorbox{gray}}

\put(0,0){\framebox(15,15)}
\put(15,0){\framebox(15,15)}
\put(30,0){\framebox(15,15)}
\put(45,0){\framebox(15,15)}
\put(60,0){\framebox(15,15)}
\put(0,15){\framebox(15,15)}
\put(15,15){\framebox(15,15)}
\put(30,15){\framebox(15,15)}
\put(45,15){\framebox(15,15)}
\put(60,15){\framebox(15,15)}
\put(0,30){\framebox(15,15)}
\put(15,30){\framebox(15,15)}
\put(30,30){\framebox(15,15)}
\put(45,30){\framebox(15,15)}
\put(60,30){\framebox(15,15)}
\put(0,45){\framebox(15,15)}
\put(15,45){\framebox(15,15)}
\put(30,45){\framebox(15,15)}
\put(45,45){\framebox(15,15)}
\put(60,45){\framebox(15,15)}
\put(0,60){\framebox(15,15)}
\put(15,60){\framebox(15,15)}
\put(30,60){\framebox(15,15)}
\put(45,60){\framebox(15,15)}
\put(60,60){\framebox(15,15)}

\put(100,63){\colorbox{gray}}
\put(100,67){\colorbox{gray}}
\put(100,72){\colorbox{gray}}
\put(104,63){\colorbox{gray}}
\put(104,67){\colorbox{gray}}
\put(104,72){\colorbox{gray}}
\put(109,63){\colorbox{gray}}
\put(109,67){\colorbox{gray}}
\put(109,72){\colorbox{gray}}

\put(115,63){\colorbox{gray}}
\put(115,67){\colorbox{gray}}
\put(115,72){\colorbox{gray}}
\put(119,63){\colorbox{gray}}
\put(119,67){\colorbox{gray}}
\put(119,72){\colorbox{gray}}
\put(124,63){\colorbox{gray}}
\put(124,67){\colorbox{gray}}
\put(124,72){\colorbox{gray}}

\put(130,63){\colorbox{gray}}
\put(130,67){\colorbox{gray}}
\put(130,72){\colorbox{gray}}
\put(134,63){\colorbox{gray}}
\put(134,67){\colorbox{gray}}
\put(134,72){\colorbox{gray}}
\put(139,63){\colorbox{gray}}
\put(139,67){\colorbox{gray}}
\put(139,72){\colorbox{gray}}

\put(145,63){\colorbox{gray}}
\put(145,67){\colorbox{gray}}
\put(145,72){\colorbox{gray}}
\put(149,63){\colorbox{gray}}
\put(149,67){\colorbox{gray}}
\put(149,72){\colorbox{gray}}
\put(154,63){\colorbox{gray}}
\put(154,67){\colorbox{gray}}
\put(154,72){\colorbox{gray}}

\put(160,63){\colorbox{gray}}
\put(160,67){\colorbox{gray}}
\put(160,72){\colorbox{gray}}
\put(164,63){\colorbox{gray}}
\put(164,67){\colorbox{gray}}
\put(164,72){\colorbox{gray}}
\put(169,63){\colorbox{gray}}
\put(169,67){\colorbox{gray}}
\put(169,72){\colorbox{gray}}

\put(100,48){\colorbox{gray}}
\put(100,52){\colorbox{gray}}
\put(100,57){\colorbox{gray}}
\put(104,48){\colorbox{gray}}
\put(104,52){\colorbox{gray}}
\put(104,57){\colorbox{gray}}
\put(109,48){\colorbox{gray}}
\put(109,52){\colorbox{gray}}
\put(109,57){\colorbox{gray}}

\put(115,48){\colorbox{gray}}
\put(115,52){\colorbox{gray}}
\put(115,57){\colorbox{gray}}
\put(119,48){\colorbox{gray}}
\put(119,52){\colorbox{gray}}
\put(119,57){\colorbox{gray}}
\put(124,48){\colorbox{gray}}
\put(124,52){\colorbox{gray}}
\put(124,57){\colorbox{gray}}

\put(130,48){\colorbox{gray}}
\put(130,52){\colorbox{gray}}
\put(130,57){\colorbox{gray}}
\put(134,48){\colorbox{gray}}
\put(134,52){\colorbox{gray}}
\put(134,57){\colorbox{gray}}
\put(139,48){\colorbox{gray}}
\put(139,52){\colorbox{gray}}
\put(139,57){\colorbox{gray}}

\put(145,48){\colorbox{gray}}
\put(145,52){\colorbox{gray}}
\put(145,57){\colorbox{gray}}
\put(149,48){\colorbox{gray}}
\put(149,52){\colorbox{gray}}
\put(149,57){\colorbox{gray}}
\put(154,48){\colorbox{gray}}
\put(154,52){\colorbox{gray}}
\put(154,57){\colorbox{gray}}

\put(160,48){\colorbox{gray}}
\put(160,52){\colorbox{gray}}
\put(160,57){\colorbox{gray}}
\put(164,48){\colorbox{gray}}
\put(164,52){\colorbox{gray}}
\put(164,57){\colorbox{gray}}
\put(169,48){\colorbox{gray}}
\put(169,52){\colorbox{gray}}
\put(169,57){\colorbox{gray}}

\put(100,33){\colorbox{gray}}
\put(100,37){\colorbox{gray}}
\put(100,42){\colorbox{gray}}
\put(104,33){\colorbox{gray}}
\put(104,37){\colorbox{gray}}
\put(104,42){\colorbox{gray}}
\put(109,33){\colorbox{gray}}
\put(109,37){\colorbox{gray}}
\put(109,42){\colorbox{gray}}

\put(115,33){\colorbox{gray}}
\put(115,37){\colorbox{gray}}
\put(115,42){\colorbox{gray}}
\put(119,33){\colorbox{gray}}
\put(119,37){\colorbox{gray}}
\put(119,42){\colorbox{gray}}
\put(124,33){\colorbox{gray}}
\put(124,37){\colorbox{gray}}
\put(124,42){\colorbox{gray}}

\put(130,33){\colorbox{gray}}
\put(130,37){\colorbox{gray}}
\put(130,42){\colorbox{gray}}
\put(134,33){\colorbox{gray}}
\put(134,37){\colorbox{gray}}
\put(134,42){\colorbox{gray}}
\put(139,33){\colorbox{gray}}
\put(139,37){\colorbox{gray}}
\put(139,42){\colorbox{gray}}

\put(145,33){\colorbox{gray}}
\put(145,37){\colorbox{gray}}
\put(145,42){\colorbox{gray}}
\put(149,33){\colorbox{gray}}
\put(149,37){\colorbox{gray}}
\put(149,42){\colorbox{gray}}
\put(154,33){\colorbox{gray}}
\put(154,37){\colorbox{gray}}
\put(154,42){\colorbox{gray}}

\put(160,33){\colorbox{gray}}
\put(160,37){\colorbox{gray}}
\put(160,42){\colorbox{gray}}
\put(164,33){\colorbox{gray}}
\put(164,37){\colorbox{gray}}
\put(164,42){\colorbox{gray}}
\put(169,33){\colorbox{gray}}
\put(169,37){\colorbox{gray}}
\put(169,42){\colorbox{gray}}

%

\put(130,18){\colorbox{gray}}
\put(130,22){\colorbox{gray}}
\put(130,27){\colorbox{gray}}
\put(134,18){\colorbox{gray}}
\put(134,22){\colorbox{gray}}
\put(134,27){\colorbox{gray}}
\put(139,18){\colorbox{gray}}
\put(139,22){\colorbox{gray}}
\put(139,27){\colorbox{gray}}

\put(145,18){\colorbox{gray}}
\put(145,22){\colorbox{gray}}
\put(145,27){\colorbox{gray}}
\put(149,18){\colorbox{gray}}
\put(149,22){\colorbox{gray}}
\put(149,27){\colorbox{gray}}
\put(154,18){\colorbox{gray}}
\put(154,22){\colorbox{gray}}
\put(154,27){\colorbox{gray}}

\put(160,18){\colorbox{gray}}
\put(160,22){\colorbox{gray}}
\put(160,27){\colorbox{gray}}
\put(164,18){\colorbox{gray}}
\put(164,22){\colorbox{gray}}
\put(164,27){\colorbox{gray}}
\put(169,18){\colorbox{gray}}
\put(169,22){\colorbox{gray}}
\put(169,27){\colorbox{gray}}

%
%
%
\put(145,3){\colorbox{gray}}
\put(145,7){\colorbox{gray}}
\put(145,12){\colorbox{gray}}
\put(149,3){\colorbox{gray}}
\put(149,7){\colorbox{gray}}
\put(149,12){\colorbox{gray}}
\put(154,3){\colorbox{gray}}
\put(154,7){\colorbox{gray}}
\put(154,12){\colorbox{gray}}

\put(160,3){\colorbox{gray}}
\put(160,7){\colorbox{gray}}
\put(160,12){\colorbox{gray}}
\put(164,3){\colorbox{gray}}
\put(164,7){\colorbox{gray}}
\put(164,12){\colorbox{gray}}
\put(169,3){\colorbox{gray}}
\put(169,7){\colorbox{gray}}
\put(169,12){\colorbox{gray}}

\put(100,0){\framebox(15,15)}
\put(115,0){\framebox(15,15)}
\put(130,0){\framebox(15,15)}
\put(145,0){\framebox(15,15)}
\put(160,0){\framebox(15,15)}
\put(100,15){\framebox(15,15)}
\put(115,15){\framebox(15,15)}
\put(130,15){\framebox(15,15)}
\put(145,15){\framebox(15,15)}
\put(160,15){\framebox(15,15)}
\put(100,30){\framebox(15,15)}
\put(115,30){\framebox(15,15)}
\put(130,30){\framebox(15,15)}
\put(145,30){\framebox(15,15)}
\put(160,30){\framebox(15,15)}
\put(100,45){\framebox(15,15)}
\put(115,45){\framebox(15,15)}
\put(130,45){\framebox(15,15)}
\put(145,45){\framebox(15,15)}
\put(160,45){\framebox(15,15)}
\put(100,60){\framebox(15,15)}
\put(115,60){\framebox(15,15)}
\put(130,60){\framebox(15,15)}
\put(145,60){\framebox(15,15)}
\put(160,60){\framebox(15,15)}

\put(-65,20){ $i$-th row $\rightarrow$}
\put(20,80){$\downarrow$}
\put(15,90){$j$-th column}

\put(35,-15){$h$}
\put(135,-15){$h'$}
\end{picture}
\end{center}
\vspace{5pt}
\caption{the pictures of $h$ and $h'$}
\label{picture:h and h'}
\end{figure}
\noindent
Then, we have $\Hess(N,h') \subset \Hess(N,h)$ by the definition \eqref{eq:DefHess(X,h)}. 
From the isomorphism \eqref{eq:AHHM} we obtain 
\begin{align*}
f_{i-1,j} \neq 0 \ \ {\rm in} \ H^*(\Hess(N,h)) \ \ \ {\rm and } \ \ \ f_{i-1,j}=0 \ \ {\rm in} \ H^*(\Hess(N,h')). 
\end{align*}
In fact, suppose for a contradiction that $f_{i-1,j} = 0$ in $H^*(\Hess(N,h))$.
Then the ideal $(f_{h(1),1},\ldots,f_{h(n),n})$ is equal to the ideal $(f_{h'(1),1},\ldots,f_{h'(n),n})$. 
It follows from \eqref{eq:AHHM} that $H^*(\Hess(N,h))$ is isomorphic to $H^*(\Hess(N,h'))$.
This contradicts the equality $\dim \Hess(N,h)=\sum_{j=1}^n (h(j)-j)=\dim \Hess(N,h')+1$.

Next, we consider intersections of a regular nilpotent Hessenberg variety and Schubert cells.
We first recall the definition of Schubert cells.
Let $G$ be the general linear group $\mbox{GL}(n, \C)$ and $B$ the standard Borel subgroup of upper-triangular invertible matrices. 
Then the flag variety $\mathcal{F}\ell(\C^n)$ can be realized as a homogeneous space $G/B$.
For a permutation $w$ in the symmetric group $S_n$, 
we define the {\bf Schubert cell} $X_w^{\circ}$ of the flag variety by $X_w^{\circ} = BwB/B$.
The Schubert cell $X_w^{\circ}$ is isomorphic to an affine space $\C^{\ell(w)}$.
It follows from \cite[Theorem~6.1]{Ty1} that the condition for $\Hess(N,h) \cap X_w^{\circ}$ being nonempty is given by 
\begin{equation} \label{eq:Ty}
\Hess(N,h) \cap X_w^{\circ} \neq \emptyset \iff w^{-1}(w(r)-1) \leq h(r) \ {\rm for \ all} \ r\in[n].
\end{equation}
The following lemma gives the geometric meaning of the permutations $w_k^{(i,j)}$ in \eqref{eq:w_k^{(i,j)}}.

\begin{lemma} \label{lemma:w_k^{(i,j)}}
Let $h:[n] \to [n]$ be a Hessenberg function with $(i,j)$-th box as a corner with $i>j$ and $h':[n] \to [n]$ a Hessenberg function obtained from $h$ by removing $(i,j)$-th box. 
Let $\{X_w^{\circ} \}$ be Schubert cells. 
Then, a permuation $w$ in $S_n$ satisfying 
\begin{align*}
\Hess(N,h) \cap X_w^{\circ} \neq \emptyset \ \ \ {\rm and} \ \ \ \Hess(N,h') \cap X_w^{\circ} = \emptyset
\end{align*}
with minimal length is given by $w_k^{(i,j)}$ in \eqref{eq:w_k^{(i,j)}} for $1 \leq k \leq i-j$.
\end{lemma}

\begin{proof}
%
Let $X_w^{\circ}$ be a Schubert cell. 
It follows from \eqref{eq:Ty} that a necessary and sufficient condition for $\Hess(N,h) \cap X_w^{\circ} \neq \emptyset$ and $\Hess(N,h') \cap X_w^{\circ} = \emptyset$ is given by 
$i-1=h'(j) < w^{-1}(w(j)-1) \leq h(j)=i$ and $w^{-1}(w(r)-1) \leq h(r)$ for $r\neq j$, that is, 
\begin{align} 
w(j)-1&=w(i), \label{eq:w(j)-1=w(i)} \\
w^{-1}(w(r)-1)& \leq h(r) \ \ \ {\rm for} \ r\neq j \label{eq:w^{-1}(w(r)-1)}. 
\end{align}
It is clear that $w_k^{(i,j)}$ satisfies \eqref{eq:w(j)-1=w(i)} and \eqref{eq:w^{-1}(w(r)-1)}.
Let $v$ be a permutation in $S_n$ satisfying \eqref{eq:w(j)-1=w(i)} and \eqref{eq:w^{-1}(w(r)-1)} with minimal length, and we prove that $v$ is a permutation $w_k^{(i,j)}$ for some $1\leq k \leq i-j$.
From the minimality of the number of inversions of $v$, we must arrange the values $v(r)$ for $r\neq j,i$ in one-line notation as a subsequence in the increasing order. 
If $v(j)=m+1, v(i)=m$ for some $m$ with $1\leq m \leq j-1$ or $i\leq m \leq n-1$, then $\ell(v)>\ell(w_k^{(i,j)})=i-j$.
This contradicts the minimality for the length of $v$.
Hence, we have $v(j)=i-k+1, v(i)=i-k$ for some $k$ with $1\leq k \leq i-j$.
This means that $v=w_k^{(i,j)}$.
\end{proof}

In summary, we can observe a geometric meaning of Theorem~\ref{theorem:main1} as follows. 
Let $\Hess(N,h)$ be a regular nilpotent Hessenberg variety.
By removing an $(i,j)$-th box from $h$, we obtain the new Hessenberg variety $\Hess(N,h')$.
Then, Lemma~\ref{lemma:w_k^{(i,j)}} tells us that $X_{w_{k}^{(i,j)}}^{\circ} \ (1 \leq k \leq i-j)$ are the minimal dimensional Schubert cells which do not intersect with the new Hessenberg variety $\Hess(N,h')$. 
Theorem~\ref{theorem:main1} now says that an alternating sum of Schubert classes $\sigma_{w_{k}^{(i,j)}}$ vanishes in $H^*(\Hess(N,h'))$ as the new relation which we do not have in $H^*(\Hess(N, h))$.

\bigskip

\noindent
\textbf{Acknowledgements}. 
The author is grateful to Hiraku Abe for fruitful discussions and comments on this paper.
The author learned the equality \eqref{eq:AbeDDO} from him. 
The author also appreciates Mikiya Masuda for his support and valuable comments on this paper.
The author is partially supported by JSPS Grant-in-Aid for JSPS Fellows: 17J04330.

\end{document}